\documentclass[12pt,bezier]{article}
\usepackage{times}
\usepackage{booktabs}
\usepackage{pifont}
\usepackage{floatrow}
\usepackage{caption}
\usepackage{mathrsfs}
\usepackage[fleqn]{amsmath}
\usepackage{amsfonts,amsthm,amssymb,mathrsfs,bbding}
\usepackage{txfonts}
\usepackage{graphics,multicol}
\usepackage{graphicx}
\usepackage{color}
\usepackage{xcolor}
\usepackage{enumerate}
\usepackage{caption}
\usepackage{cite}
\usepackage{latexsym,bm}
\usepackage{indentfirst}
\usepackage{mathtools}
\evensidemargin=\oddsidemargin\topmargin=-1.5cm

\newtheorem{thm}{Theorem}[section]

\newtheorem{lem}{Lemma}[section]
\newtheorem{cor}{Corollary}[section]

\newtheorem{remark}{Remark}
\newtheorem{example}{Example}

\theoremstyle{definition}

\addtocounter{section}{0}

\begin{document}
\title{On graphs with three or four distinct normalized Laplacian eigenvalues\footnote{This work is supported
by NSFC (Nos. 11671344, 11531011, 11501486).}}
\author{{\small Xueyi Huang, \ \ Qiongxiang Huang\footnote{
Corresponding author.}\setcounter{footnote}{-1}\footnote{
\emph{E-mail address:} huangxymath@gmail.com (X. Huang), huangqx@xju.edu.cn (Q. Huang).} }\\[2mm]\scriptsize
College of Mathematics and Systems Science,
\scriptsize Xinjiang University, Urumqi, Xinjiang 830046, P. R. China}

\date{}
\maketitle
{\flushleft\large\bf Abstract}
In this paper,  we characterize all connected graphs with exactly three distinct normalized Laplacian eigenvalues of which one  is equal to $1$, determine all connected bipartite graphs with at least one vertex of degree $1$ having exactly four distinct normalized Laplacian eigenvalues, and find all unicyclic graphs with three or four distinct normalized Laplacian eigenvalues.
\vspace{0.1cm}
\begin{flushleft}
\end{flushleft}

\section{Introduction}\label{s-1}
Let $G$ be a simple  undirected graph  on $n$ vertices with $m$ edges. The  \emph{adjacency matrix} $A=(a_{uv})$ of $G$ is the $n\times n$ matrix with rows and columns indexed by the vertices, where $a_{uv} = 1$ if $u$ is adjacent to $v$, and $0$ otherwise. Let $D=\mathrm{diag}(d_{v_1},d_{v_2},\ldots,d_{v_n})$ denote the \emph{diagonal degree matrix} of $G$. The well-known \emph{Laplacian matrix}   of  $G$ is defined as $L=D-A$. The \emph{normalized Laplacian matrix} (\emph{$\mathcal{L}$-matrix} for short) of $G$ is the $n \times n$ matrix $\mathcal{L} = (\ell_{uv})$ with
\begin{equation*}
\ell_{uv}=\left\{
\begin{array}{ll}
1& \mbox{if $u=v$, $d_u\neq 0$,}\\
-1/\sqrt{d_ud_v}&\mbox{if $u$ is adjacent to $v$,}\\
0&\mbox{otherwise.}
\end{array}
\right.
\end{equation*}
Clearly, if $G$ has no isolated vertices then $\mathcal{L}=D^{-\frac{1}{2}}LD^{-\frac{1}{2}}=I-D^{-\frac{1}{2}}AD^{-\frac{1}{2}}$. Since $\mathcal{L}$ is positive semidefinite, all its eigenvalues are nonnegative. These eigenvalues are also called the \emph{normalized Laplacian eigenvalues} ($\mathcal{L}$-\emph{eigenvalues} for short) of $G$.  We denote by $\lambda_1>\lambda_2>\cdots>\lambda_t$  all the distinct $\mathcal{L}$-eigenvalues of $G$ with multiplicities $m_1,m_2,\ldots,m_t$ ($\sum_{i=1}^tm_i=n$), respectively. All these $\mathcal{L}$-eigenvalues together with their multiplicities are called the \emph{normalized Laplacian spectrum} ($\mathcal{L}$-\emph{spectrum} for short) of $G$ denoted by $\mathrm{Spec}_{\mathcal{L}}(G)=\big\{[\lambda_1]^{m_1},[\lambda_2]^{m_2},\ldots,[\lambda_t]^{m_t}\big\}$. With respect to  adjacency matrix, the  \emph{adjacency spectrum} of $G$ can be similarly defined and denoted by $\mathrm{Spec}_A(G)$.

A connected graph on $n$ vertices with $m$ edges is called a \emph{$k$-cyclic graph}  if $k=m-n+1$. In particular,  the notions \emph{tree} and \emph{unicyclic graph} are respectively defined as the $k$-cyclic graph with $k=0$ and $k=1$.

Throughout this paper,  we denote the \emph{neighborhood} of a vertex $v\in V(G)$ by $N_{G}(v)$, the disjoint union of graphs $G$ and $H$ by $G\cup H$, the complete graph on  $n$ vertices by $K_n$ and the complete multipartite graph with $s$ parts of sizes $n_1,\ldots,n_s$  by $K_{n_1,\ldots,n_s}$. Also,  the  $n\times n$ identity matrix,  the $n\times 1$ all-ones  vector and  the $n\times n$ all-ones matrix will
be denoted by  $I_n$, $\mathbf{j}_n$  and $J_n$, respectively.

Connected graphs with  few distinct eigenvalues have been investigated frequently for several graph matrices over the past two decades, such as the adjacency matrix \cite{Bridges,Caen,Cheng,Cioaba,Cioaba1,Dam1,Dam3,Dam7,Dam2,Dam5,Dam4,Doob,Huang,Lima,Muzychuk,Rowlinson,Shrikhande}, the Laplacian matrix \cite{Dam8,Mohammadian,Wang}, the signless Lapalcian matrix \cite{Ayoobi}, the Seidel matrix \cite{Seidel}, and the universal adjacency matrix \cite{Haemers}. One of the reason is that such graphs in general have pretty combinatorial properties and a rich structure \cite{Dam2}. With regard to normalized Laplacian matrix,  Cavers \cite{Cavers} characterized all connected graphs with at least one vertex of degree $1$ having three distinct $\mathcal{L}$-eigenvalues, Van Dam et al. \cite{Dam9} gave all connected triangle-free graphs (particularly, bipartite graphs) with three distinct $\mathcal{L}$-eigenvalues, and Braga et al. \cite{Braga} determined all trees with four or five distinct $\mathcal{L}$-eigenvalues.

In this paper,  we characterize all connected graphs with exactly three distinct $\mathcal{L}$-eigenvalues of which one  $\mathcal{L}$-eigenvalue is $1$, determine all connected bipartite graphs with at least one vertex of degree $1$ that have exactly four distinct $\mathcal{L}$-eigenvalues, and find all unicyclic graphs with three or four distinct $\mathcal{L}$-eigenvalues.

\section{Main tools}\label{s-2}
First of all, we recall some basic results about $\mathcal{L}$-eigenvalues.
\begin{lem}(See \cite{Chung,Dam9}.)\label{lem-2-2}
Let $\lambda_1\geq \lambda_2\geq\cdots\geq\lambda_n$ ($n\geq 2$) be all the $\mathcal{L}$-eigenvalues of $G$. Then $G$ has the following properties.
\begin{enumerate}[(i)]
\vspace{-0.2cm}
\item $\lambda_n = 0$,
\vspace{-0.3cm}
\item $\sum_i\lambda_i\leq n$ with equality holding if and only if $G$ has no isolated vertices,
    \vspace{-0.3cm}
\item $\lambda_{n-1}\leq \frac{n}{n-1}$ with equality holding if and only if $G$ is a complete graph on $n$ vertices,
    \vspace{-0.3cm}
\item $\lambda_{n-1}\leq 1$ if $G$ is non-complete,
\vspace{-0.3cm}
\item $\lambda_{1}\geq \frac{n}{n-1}$ if $G$ has no isolated vertices,
\vspace{-0.3cm}
\item $\lambda_{n-1} > 0$ if $G$ is connected. If $\lambda_{n-i+1} = 0$ and $\lambda_{n-i}\neq 0$, then $G$ has exactly $i$ connected components,
    \vspace{-0.3cm}
\item The $\mathcal{L}$-spectrum of $G$ is the union of the  $\mathcal{L}$-spectra of its connected components,
\vspace{-0.3cm}
\item $\lambda_i\leq 2$ for all $i$, with $\lambda_1 = 2$ if and only if some connected component of $G$ is a non-trivial bipartite graph,
\vspace{-0.3cm}
\item $G$ is bipartite if and only if $2 -\lambda_i$ is an  $\mathcal{L}$-eigenvalue of $G$ for each $i$.
\end{enumerate}
\end{lem}

Two $n\times n$ real symmetric matrices $M$ and $N$ are said to be \emph{congruent} if there exists an invertible matrix $S\in \mathbb{R}^{n\times n}$ such that
$S^TMS = N$. The well-known \emph{Sylvester's law of inertia} states that two congruent real symmetric matrices have the same numbers of positive, negative, and zero eigenvalues. As Lemma \ref{lem-2-2} (iv) suggests that $\lambda_{n-1}\leq 1$ if $G$ is not a complete graph, we now characterize all the connected graphs attaining this bound.
\begin{cor}\label{cor-2-0}
Let $G$ be a connected non-complete graph, and $\lambda_{n-1}$ the second least $\mathcal{L}$-eigenvalue of $G$. Then $\lambda_{n-1}\leq 1$ with equality holding if and only if $G$ is a complete multipartite graph.
\end{cor}
\begin{proof}
Suppose $A$ and $D$ are the adjacency matrix and diagonal degree matrix of $G$, respectively. Let $A^*=D^{-\frac{1}{2}}AD^{-\frac{1}{2}}$. The $\mathcal{L}$-matrix of $G$ is $\mathcal{L}=I-A^*$. Therefore, $\lambda_{n-1}=1$ if and only if the second largest eigenvalue of $A^*$ is equal to $0$, which is the case if and only if the second largest eigenvalue of $A$ is equal to $0$ because $A^*$ and $A$ have the same number of positive, negative, and zero eigenvalues due to $A^*$ and $A$ are congruent. It is well known that  a connected graph has $0$ as its second largest  (adjacency) eigenvalue  if and only if it is a complete multipartite graph (except the complete graph). The result follows.
\end{proof}

The following lemma suggests that some special $\mathcal{L}$-eigenvlaues are related to
some local properties of graphs.

\begin{lem}\label{lem-2-3}
Let $G$ be a graph on $n$ vertices, and let $X=\{v_1,v_2,\ldots,v_p\}$ $(p\geq 2)$ be a set of vertices such that $N_G(v_1)\setminus X=N_G(v_2)\setminus X=\cdots=N_G(v_p)\setminus X=\{u_1,u_2,\ldots,u_q\}$ $(q\geq 1)$. We have
\begin{enumerate}[(i)]
\vspace{-0.2cm}
\item if $X$ is an independent set of $G$, then $1$ is an $\mathcal{L}$-eigenvalue of $G$ with multiplicities at least $p-1$;
    \vspace{-0.3cm}
\item if $X$ is a clique of $G$, then $\frac{p+q}{p+q-1}$ is an $\mathcal{L}$-eigenvalue of $G$ with multiplicities at least $p-1$.
\end{enumerate}
\end{lem}
\begin{proof}
For any fixed $i$ ($1\leq i\leq p-1$), let $x_i\in\mathbb{R}^n$ be the vector defined by $x_i(v_i)=1$, $x_i(v_p)=-1$ and  $x_i(v)=0$ for each $v\not\in\{v_i,v_p\}$ in $G$. To prove (i), it suffices to verify that each $x_i$ is an eigenvector of $G$ corresponding to the $\mathcal{L}$-eigenvalue $1$.  In fact, under the assumption of (i),  we have $N_G(v_1)=N_G(v_2)=\cdots=N_G(v_p)=\{u_1,u_2,\ldots,u_q\}$ since $X$ is independent set. Recall that $\mathcal{L}=I-D^{-\frac{1}{2}}AD^{-\frac{1}{2}}$, the following equalities conform that $\mathcal{L}x_i=1\cdot x_i$:
\begin{equation*}
\left\{
\begin{aligned}
&x_i(v_j)-\sum_{v\sim v_j}\frac{1}{\sqrt{d_{v}d_{v_j}}}x_i(v)=x_i(v_j)-\sum_{l=1}^{q}\frac{1}{\sqrt{d_{u_l}d_{v_j}}}x_i(u_l)=1\cdot x_i(v_j) ~\text{for $1\leq j\leq p$},\\
&x_i(u_j)-\sum_{v\sim u_j}\frac{1}{\sqrt{d_{v}d_{u_j}}}x_i(v)=0-\Bigg(\frac{1}{\sqrt{d_{v_i}d_{u_j}}}-\frac{1}{\sqrt{d_{v_p}d_{u_j}}}\Bigg)=1\cdot x_i(u_j) ~\text{for $1\leq j\leq q$},\\
&x_i(u)-\sum_{v\sim u}\frac{1}{\sqrt{d_{v}d_{u}}}x_i(v)=0=1\cdot x_i(u)~\text{for $u\not\in\{v_1,\ldots,v_p,u_1,\ldots,u_q\}$}.
\end{aligned}
\right.
\end{equation*}
Hence $1$ is an $\mathcal{L}$-eigenvalue of $G$ with multiplicities at least $p-1$.

Under the assumption of (ii), one can similarly verify that the same $x_i\in\mathbb{R}^n$ defined above is also an 
 eigenvector of $G$ with respect to the $\mathcal{L}$-eigenvalue $\frac{p+q}{p+q-1}$ for $i=1,2,\ldots,p-1$.
\end{proof}
By Lemma \ref{lem-2-2} (ix) and Lemma \ref{lem-2-3} (i), we obtain the following corollary immediately.
\begin{cor}\label{cor-2-1}
Let $G$ be a bipartite graph. If there are two vertices in $G$ share the same neighborhood, then the number of distinct $\mathcal{L}$-eigenvalues of $G$ must be odd.
\end{cor}

Let $M$ be a real symmetric matrix whose all distinct eigenvalues are $\lambda_1,\ldots,\lambda_s$. Then $M$ has the \emph{spectral decomposition} $M=\lambda_1P_1+\cdots+\lambda_sP_s$, where $P_i=\mathbf{x}_{1}\mathbf{x}_{1}^T+\cdots+\mathbf{x}_{r}\mathbf{x}_{r}^T$ if the eigenspace $\mathcal{E}(\lambda_i)$ has $\{\mathbf{x}_1,\ldots,\mathbf{x}_r\}$ as an orthonormal basis. It is well known that  for any real polynomial $f(x)$, we have $f(M)=f(\lambda_1)P_1+\cdots+f(\lambda_s)P_s$.

Note that  for any graph $G$, we have $\mathcal{L}\cdot D^\frac{1}{2}\mathbf{j}=0\cdot D^\frac{1}{2}\mathbf{j}$, and so  $D^\frac{1}{2}\mathbf{j}$ is  an eigenvector  with respect to the $\mathcal{L}$-eigenvalue $0$ of $G$. Moreover, if $G$ is connected, from Lemma \ref{lem-2-2} (vi) we know that $0$ is a simple $\mathcal{L}$-eigenvalue of $G$. Keeping these facts in mind, one can easily deduce the following lemma by using the method of spectral decomposition.

\begin{lem}(See \cite{Cavers}.)\label{lem-2-1}
Let $G$ be a connected graph on $n\geq3$ vertices with $m$ edges and fix $2\leq s\leq n$. Then $G$ has exactly $s$ distinct $\mathcal{L}$-eigenvalues, namely $\lambda_1,\lambda_2,\ldots,\lambda_{s-1}$ and $0$, if and only if there exists $s-1$ distinct nonzero
numbers $\lambda_1,\lambda_2,\ldots,\lambda_{s-1}$ such that
\begin{equation}\label{equ-1}
\prod_{i=1}^{s-1}(\mathcal{L}-\lambda_iI)=(-1)^{s-1}\Bigg(\prod_{i=1}^{s-1}\lambda_i\Bigg)\frac{D^{\frac{1}{2}}JD^{\frac{1}{2}}}{2m},
\end{equation}
where $D$ is the diagonal degree matrix  of $G$.
\end{lem}

By Lemma \ref{lem-2-1}, we have the following result.
\begin{lem}\label{lem-2-4}
Let $G$ be a connected graph with $m$ edges and $\mathrm{Spec}_{\mathcal{L}}(G)=\{[\alpha]^{m_1},[\beta]^{m_2},[0]^1\}$ ($2\geq \alpha>\beta>0$). Then $\beta\leq 1$, and one of the following holds:
\begin{enumerate}[(i)]
\vspace{-0.2cm}
\item $\beta=1$, and each pair of non-adjacent vertices in $G$ have the same neighbors;
\vspace{-0.3cm}
\item $\beta<1$, and $-\frac{2m(\alpha-1)(\beta-1)}{\alpha\beta}\geq d_u-d_v\geq \frac{2m(\alpha-1)(\beta-1)}{\alpha\beta}$ holds for any two non-adjacent vertices $u,v$ in $G$.
\end{enumerate}
\end{lem}
\begin{proof}
By the assumption, $\mathrm{Spec}_{\mathcal{L}}(G)=\{[\alpha]^{m_1},[\beta]^{m_2},[0]^1\}$,  where $2\geq\alpha>\beta>0$. By Lemma \ref{lem-2-1}, we have 
\begin{equation}\label{equ-0}
(\mathcal{L}-\alpha I)(\mathcal{L}-\beta I)=\frac{\alpha\beta}{2m}D^{\frac{1}{2}}JD^{\frac{1}{2}}.
\end{equation}
Since $\mathcal{L}=I-D^{-\frac{1}{2}}AD^{-\frac{1}{2}}$, (\ref{equ-0}) can be written as
\begin{equation}\label{equ-0-1}
\Big(D^{-\frac{1}{2}}AD^{-\frac{1}{2}}\Big)^2+(\alpha+\beta-2) D^{-\frac{1}{2}}AD^{-\frac{1}{2}}+(\alpha-1)(\beta-1)I=\frac{\alpha\beta}{2m}D^{\frac{1}{2}}JD^{\frac{1}{2}}.
\end{equation}
By considering the $(u,u)$-entry and $(u,v)$-entry ($u\neq v$) at both sides of (\ref{equ-0-1}), we obtain that
\begin{equation}\label{equ-3}
\sum_{w\sim u}\frac{1}{d_w}=\frac{\alpha\beta}{2m}d_u^2-(\alpha-1)(\beta-1)d_u,
\end{equation}
and
\begin{equation}\label{equ-4}
\sum_{\small\begin{smallmatrix}w\sim u\\ w\sim v\end{smallmatrix}}\frac{1}{d_w}=
\begin{cases}
\displaystyle-(\alpha+\beta-2)+\frac{\alpha\beta}{2m}d_ud_v&\mbox{if $u\sim v$},\\
\vspace{0.1cm}
\displaystyle\frac{\alpha\beta}{2m}d_ud_v &\mbox{if $u\nsim v$}.\\
\end{cases}
\end{equation}
Let $u,v$ be a pair of non-adjacent vertices in $G$. Then from (\ref{equ-3}) and (\ref{equ-4}) we deduce that 
\begin{equation}\label{equ-5}
\left\{
\begin{aligned}
\frac{\alpha\beta}{2m}d_u^2-(\alpha-1)(\beta-1)d_u=\sum_{w\sim u}\frac{1}{d_w}\geq \sum_{\small\begin{smallmatrix}w\sim u\\ w\sim v\end{smallmatrix}}\frac{1}{d_w}=\frac{\alpha\beta}{2m}d_ud_v,\\
\frac{\alpha\beta}{2m}d_v^2-(\alpha-1)(\beta-1)d_v=\sum_{w\sim v}\frac{1}{d_w}\geq \sum_{\small\begin{smallmatrix}w\sim u\\ w\sim v\end{smallmatrix}}\frac{1}{d_w}=\frac{\alpha\beta}{2m}d_ud_v.\\
\end{aligned}
\right.
\end{equation}
This implies that 
\begin{equation}\label{equ-0-3}
-\frac{2m(\alpha-1)(\beta-1)}{\alpha\beta}\geq d_u-d_v\geq \frac{2m(\alpha-1)(\beta-1)}{\alpha\beta}.
\end{equation}
Since $\alpha>1$ by Lemma \ref{lem-2-2} (v), we must have $\beta\leq 1$ by (\ref{equ-0-3}), which is consistent with Lemma \ref{lem-2-2} (iv). Particularly, if $\beta=1$, from (\ref{equ-0-3}) we may  conclude that $d_u=d_v$, and moreover,  $u$ and $v$ must share the same neighborhood in $G$ by (\ref{equ-5}). 

We complete the proof.
\end{proof}
\begin{remark}\label{rem-0}
\emph{It is worth mentioning that Lemma \ref{lem-2-4} (i) can be also  deduced from Corollary \ref{cor-2-0}.}
\end{remark}
At the end of this section, we characterize all connected graphs with exactly three distinct $\mathcal{L}$-eigenvalues of which one is equal to $1$.
\begin{thm}\label{thm-2-1}
Let $G$ be a connected graph of order $n\geq 3$. Then $G$ has exactly three distinct $\mathcal{L}$-eigenvalues of which one $\mathcal{L}$-eigenvalue is  $1$ if and only if $G\cong K_{s,n-s}$ with $1\leq s\leq n-1$, or $G\cong K_{n_1,n_2,\ldots,n_r}$ with $n_1=n_2=\cdots=n_r=\frac{n}{r}$  and $3\leq r\leq n-1$.
\end{thm}
\begin{proof}
By the assumption, we can suppose that $\mathrm{Spec}_{\mathcal{L}}(G)=\{[\alpha]^{m_1},[1]^{m_2},[0]^1\}$ ($\alpha>1$). By Lemma \ref{lem-2-4} (i),  each pair of non-adjacent vertices in $G$ share the same neighborhood. For any  $u\in V(G)$, let $[u]$ denote the set of vertices which are not adjacent to $u$ in $G$. Then each vertex in $[u]$ has the same neighborhood as $u$. Therefore,  $[u]$ forms an independent set of $G$ and all vertices in $[u]$ share the same neighborhood $V(G)\setminus [u]$.  Hence, by the arbitrariness of $u$, we may conclude that $G$ is a complete multipartite graph, say $G=K_{n_1,n_2,\ldots,n_r}$, where $n_1+n_2+\cdots+n_r=n$,  and $2\leq r\leq n-1$ because $G$ is connected and cannot be a complete graph (which has only two distinct $\mathcal{L}$-eigenvalues). If $r=2$, then $G$ is a complete bipartite graph and our result follows. Now assume that $3\leq r\leq n-1$. Let $V_i$ ($1\leq i\leq r$) be the $i$-th part of $V(G)$ with $|V_i|=n_i$. If $r=3$,  for any three vertices  $u\in V_1$, $v\in V_2$ and $w\in V_3$, we have  $u\sim v$, $u\sim w$ and $v\sim w$. By applying (\ref{equ-4}) to these three pairs of adjacent vertices, we get
 \begin{equation}\label{equ-5-0}
\left\{
\begin{aligned}
\frac{n_3}{n-n_3}=1-\alpha+\frac{\alpha}{2m}(n-n_1)(n-n_2)\\
\frac{n_2}{n-n_2}=1-\alpha+\frac{\alpha}{2m}(n-n_1)(n-n_3)\\
\frac{n_1}{n-n_1}=1-\alpha+\frac{\alpha}{2m}(n-n_2)(n-n_3)
\end{aligned}
\right.~\Rightarrow~
\left\{
\begin{aligned}
n_3-n_2=(1-\alpha)(n_2-n_3)\\
n_2-n_1=(1-\alpha)(n_1-n_2)
\end{aligned}
\right.
\end{equation}
which gives that $n_1=n_2=n_3$ because $\alpha\neq 2$ due to $G$ is not a bipartite graph. For $4\leq r\leq n-1$,  we have to deal with it  by the way of contradiction since we cannot obtain  a  similar symmetric relation as in (\ref{equ-5-0}). Assume that not all  $n_i$'s are equal, say $n_1\neq n_2$. Taking $u_1\in V_1$, $u_2\in V_2$ and $u_i\in V_i$ ($3\leq i\leq r$), then $u_1\sim u_i$ and $u_2\sim u_i$. By applying (\ref{equ-4}) to these two pairs of adjacent vertices, we get
 \begin{equation}\label{equ-5-1}
\left\{
\begin{aligned}
\sum_{l\neq 1,i}\frac{n_l}{n-n_l}=1-\alpha+\frac{\alpha}{2m}(n-n_1)(n-n_i)\\
\sum_{l\neq 2,i}\frac{n_l}{n-n_l}=1-\alpha+\frac{\alpha}{2m}(n-n_2)(n-n_i)\\
\end{aligned}
\right.~\Rightarrow~
\frac{n(n_1-n_2)}{(n-n_1)(n-n_2)}=\frac{\alpha}{2m}(n_1-n_2)(n-n_i).
\end{equation}
Since $n_1\neq n_2$, from (\ref{equ-5-1}) we have
\begin{equation}\label{equ-6}
\frac{\alpha}{2m}=\frac{n}{(n-n_1)(n-n_2)(n-n_i)}.
\end{equation} Thus we obtain $n_3=\cdots=n_r$ by the arbitrariness of $i$. Consequently, we see that $n_1\neq n_3$ or $n_2\neq n_3$, say $n_1\neq n_3$. By exchanging the roles of $n_2$ and $n_3$, similarly   as above arguments, we deduce that $n_2=n_4=\cdots=n_r$. Therefore, we have $n_2=n_3=\cdots=n_r$. Furthermore, putting $u=u_1\in V_1$ in (\ref{equ-3}), we obtain $\frac{\alpha}{2m}(n-n_1)^2=\frac{1}{n-n_2}\cdot(n-n_1)$, i.e.,
\begin{equation}\label{equ-6-0}
\frac{\alpha}{2m}=\frac{1}{(n-n_1)(n-n_2)}.
\end{equation}
Combining (\ref{equ-6}) and (\ref{equ-6-0}), we deduce that $n_i=0$ for $3\leq i\leq r$, which is impossible. Therefore, we have  $n_1=n_2=\cdots=n_r=\frac{n}{r}$, and our result follows.

Conversely, by simple computation we obtain that
\begin{equation}\label{equ-7}
\begin{cases}
\mathrm{Spec}_{\mathcal{L}}(K_{s,n-s})=\{[2]^1, [1]^{n-2}, [0]^1\},\\
\mathrm{Spec}_{\mathcal{L}}(K_{\frac{n}{r},\frac{n}{r},\ldots,\frac{n}{r}})=\{[\frac{r}{r-1}]^{r-1}, [1]^{n-r}, [0]^1\},\\
\end{cases}
\end{equation}
where $1\leq s\leq n-1$ and $3\leq r\leq n-1$. It follows our result.
\end{proof}
\begin{cor}\label{cor-0-1}
Let $G$ be a connected graph of order $n\geq 3$. Then $G$ has $\mathcal{L}$-spectrum $\mathrm{Spec}_{\mathcal{L}}(G)=\{[\alpha]^1,[\beta]^{n-2},[0]^1\}$  ($\alpha>\beta>0$) if and only if $\alpha=2$ and $\beta=1$ if and only if $G$ is a complete bipartite graph.
\end{cor}
\begin{proof}
Suppose that $\mathrm{Spec}_{\mathcal{L}}(G)=\{[\alpha]^1,[\beta]^{n-2},[0]^1\}$. By considering the trace of $\mathcal{L}$, we have $\alpha+(n-2)\beta=n$, implying that $\beta=\frac{n-\alpha}{n-2}\geq 1$ due to $\alpha\leq 2$ by Lemma \ref{lem-2-2} (viii).  Again by Lemma \ref{lem-2-2} (iv), we get $\beta=1$, and so $\alpha=2$. Thus $G$ is a complete bipartite graph by Theorem \ref{thm-2-1}. Conversely, the $\mathcal{L}$-spectrum of a complete bipartite graph $G$ is  of the form  $\mathrm{Spec}_{\mathcal{L}}(G)=\{[\alpha]^1,[\beta]^{n-2},[0]^1\}$ by (\ref{equ-7}), as required. 
\end{proof}
\begin{remark}\label{rem-0-1}
\emph{Note that Corollary \ref{cor-0-1} has been obtained by Van Dam and Omidi \cite{Dam9}. In fact, they have determined all connected graphs with three distinct $\mathcal{L}$-eigenvalues of which two are simple.}
\end{remark}

\section{Bipartite graphs with   four distinct $\mathcal{L}$-eigenvalues}
In this section, we focus on connected bipartite graphs with four distinct $\mathcal{L}$-eigenvalues, and determine all such graphs with at least one vertex of degree $1$. First of all, we need some concepts and results coming from combinatorial design theory for later use.

A \emph{balanced incomplete block design} (\emph{BIBD} for short) is a pair $(V,\mathcal{B})$ where $V$ is a $v$-set and $\mathcal{B}$ is a collection of $b$ $k$-subsets (\emph{blocks}) of $V$ such that each element of $V$ is contained in exactly $r$ blocks,   and  each pair of elements of $V$ is simultaneously contained in $\lambda$ blocks (see \cite{Cvetkovic}). The integers $(v,b,r,k,\lambda)$ are called the \emph{parameters} of the BIBD $(V,\mathcal{B})$. The \emph{complement} of $(V,\mathcal{B})$ is $(V,\overline{\mathcal{B}})$, where $\overline{\mathcal{B}}=\{V\setminus B: B\in \mathcal{B}\}$. Clearly, $(V,\overline{\mathcal{B}})$ is a BIBD with paramenters $(v,b,b-r,v-k,b-2r+\lambda)$. In the case $r=k$ (and then $v=b$) the BIBD $(V,\mathcal{B})$ is called \emph{symmetric} with parameters $(v,k,\lambda)$. In particular, the complement of a symmetric BIBD with parameters $(v,k,\lambda)$ is also a symmetric BIBD, which has parameters $(v,v-k,v-2k+\lambda)$. 

Let  $(V,\mathcal{B})$ be a BIBD with parameters $(v,b,r,k,\lambda)$.  The \emph{incidence matrix} of  $(V,\mathcal{B})$ is a  $v\times b$ matrix $C=(c_{ij})$, in which $c_{ij}=1$ when the $i$-th element $v_i$ of $V$ occurs in the $j$-th block $B_j$ of $\mathcal{B}$ and $c_{ij}=0$ otherwise. The \emph{incidence graph} of  $(V,\mathcal{B})$ is the bipartite graph on  $b+v$ vertices with the bipartition $V\cup \mathcal{B}$ in which $v_i\in V$ and $B_j\in \mathcal{B}$ are adjacent if and only if $v_i\in B_j$. As   shown in \cite{Cvetkovic} (pp. 165--167),  the incidence graph has adjacency spectrum $\big\{[\sqrt{rk}]^1,[\sqrt{r-\lambda}]^{v-1},[0]^{b-v},$ $[-\sqrt{r-\lambda}]^{v-1},[-\sqrt{rk}]^1\big\}$. In particular, if $(V,\mathcal{B})$ is symmetric,  the incidence graph is a $k$-regular bipartite graph with adjacency spectrum 
\begin{equation}\label{equ-4-4}
\big\{[k]^1,[\sqrt{k-\lambda}]^{v-1},[-\sqrt{k-\lambda}]^{v-1},[-k]^1\big\}.
\end{equation}
Conversely, a connected regular bipartite graph with four distinct (adjacency) eigenvalues  is the incidence graph of a symmetric BIBD (see \cite{Brouwer}, Proposition 14.1.3).

A square matrix $H$ of order $n$ whose entries are $+1$ or $-1$ is called  a \emph{Hadamard matrix of order $n$} provided that its rows are pairwise orthogonal, in other words $HH^T=nI$. It is well known that a Hadamard matrix of order $n$  exists only if  $n=1$, $2$ or $4t$, where $t$ is a positive integer \cite{Seberry}.  Multiplying any row (column) of a Hadamard matrix by $-1$, or permuting rows (columns) of  a Hadamard matrix, the result is also a Hadamard matrix. Two Hadamard matrices are said to be \emph{equivalent} if one can be obtained from the other by a sequence these operations. It is easy to see that every Hadamard matrix is equivalent to a Hadamard matrix that has every element of its first row and column $+1$, which is called a \emph{normalized Hadamard matrix}. Clearly, in a normalized Hadamard matrix of order $4t$, every row (column) except the first contains $+1$ and $-1$ exactly $2t$ times each, and further, $+1$ (resp. $-1$) in any row (column) except the first overlap with $+1$ (resp. $-1$) in each other row (column) except the first exactly $t$ times each.

Assume that there exists a  Hadamard matrix $H$ of order $4t$. Without loss of generality, suppose that $H$ is normalized. Remove the first row and column of $H$ and replace every $-1$ in the resulting matrix by a $0$. The final  $(4t-1)\times(4t-1)$ matrix $C$ can be viewed as the incidence matrix of a symmetric BIBD with parameters $(4t-1,2t-1,t-1)$ by above arguments. Conversely, given a symmetric BIBD with parameters $(4t-1,2t-1,t-1)$, a Hadamard matrix could be constructed  by reversing above process. For this reason, a symmetric BIBD with parameters $(4t-1,2t-1,t-1)$ is called a \emph{Hadamard design of dimension $t$}. Therefore, a Hadamard matrix corresponds to a Hadamard design naturally. 

Now we begin to consider connected bipartite graphs with four distinct $\mathcal{L}$-eigenvalues. Suppose that $G$ is a connected bipartite graph on $n$ vertices  $m$ edges with the bipartition $V(G)=V_1\cup V_2$, where $|V_i|=n_i$ for $i=1,2$, and $n_1\leq n_2$. Then the adjacency matrix $A$ and the  diagonal degree matrix $D$ of $G$ can be respectively written as 
\begin{equation*}
A=
\left(\begin{matrix}
0& B\\
B^T& 0\\
\end{matrix}\right)
\begin{matrix}
V_1\\
V_2\\
\end{matrix}
~~\mbox{and}~~D=\left(\begin{matrix}D_1&0\\ 0& D_2\end{matrix}\right),
\end{equation*}
where $D_i$ corresponds to $V_i$ for $i=1,2$. The $\mathcal{L}$-matrix of $G$ is of the form
\begin{equation*}
\mathcal{L}=I-D^{-\frac{1}{2}}AD^{-\frac{1}{2}}=I-\left(\begin{matrix} 0& D_1^{-\frac{1}{2}}BD_2^{-\frac{1}{2}}\\ D_2^{-\frac{1}{2}}B^TD_1^{-\frac{1}{2}} &0\end{matrix}\right)=I-\left(\begin{matrix} 0& B^*\\ B^{*T} &0\end{matrix}\right)=I-A^*,
\end{equation*}
where $B^{*}=D_1^{-\frac{1}{2}}BD_2^{-\frac{1}{2}}$ and $A^*=\left(\begin{matrix} 0& B^{*}\\ B^{*T} &0\end{matrix}\right)$. Then $A^{*2}=\left(\begin{matrix} B^{*}B^{*T}& 0\\ 0& B^{*T}B^{*}\end{matrix}\right)$ and $|\xi I-A^{*2}|=|\xi I_{n_1}-B^{*}B^{*T}|\cdot|\xi I_{n_2}- B^{*T}B^{*}|=\xi^{n_2-n_1}|\xi I_{n_1}-B^{*}B^{*T}|^2$. Assume that $\xi_1,\ldots,\xi_{n_1}$ are all the eigenvalues of $B^{*}B^{*T}$. Note that if $\xi$ is an eigenvalue of $B^{*}B^{*T}$ then $\pm\sqrt{\xi}$ must be eigenvalues of $A^{*}$. Then all the eigenvalues of $A^*$ are given by $\pm\sqrt{\xi_1},\ldots,\pm\sqrt{\xi_{n_1}}$, and $(n_2-n_1)$'s  $0$.  Therefore, we have $\mathrm{Spec}_{\mathcal{L}}(G)=\{[1\pm \sqrt{\xi_1}],\ldots,[1\pm \sqrt{\xi_{n_1}}],[1]^{n_2-n_1}\}$, which implies that $n_1=n_2$ if $G$ has  an even number of distinct $\mathcal{L}$-eigenvalues.

Now we are in a position to prove the main result of this section. 
\begin{thm}\label{thm-4-1}
Let $G$ a connected bipartite graph with at least one vertex of degree $1$. Then $G$ has exactly four distinct $\mathcal{L}$-eigenvalues if and only if $G$ is the graph obtained from $G'\cup K_2$ by joining one vertex of $K_2$ to one part of $G'$, where $G'=K_2$, or $G'$  is the incidence graph of the complement of a Hadamard design.
\end{thm}
\begin{proof}
Let $G$ be a connected bipartite graph on $n$ vertices $m$ edges with at least one vertex of degree $1$. Suppose that $V(G)=V_1\cup V_2$ is the bipartition of $G$, where $|V_i|=n_i$ for $i=1,2$. If $G$ has  four distinct $\mathcal{L}$-eigenvalues,  we must have $n_1=n_2=\frac{n}{2}$ by above arguments, and furthermore, we can assume that $\mathrm{Spec}_{\mathcal{L}}(G)=\{[2]^1,[2-\alpha]^{\frac{n-2}{2}},[\alpha]^{\frac{n-2}{2}},[0]^1\}$ ($0<\alpha<1$) by Lemma \ref{lem-2-2}. Set
\begin{equation*}
\mathbf{x}_1=\frac{1}{\sqrt{2m}}D^{\frac{1}{2}}\mathbf{j}_n=\frac{1}{\sqrt{2m}}\left(
\begin{matrix}
D_1^{\frac{1}{2}}\mathbf{j}_{\frac{n}{2}}\\
D_2^{\frac{1}{2}}\mathbf{j}_{\frac{n}{2}}
\end{matrix} \right)
\mbox{ and } \mathbf{x}_2=\frac{1}{\sqrt{2m}}\left(
\begin{matrix}
D_1^{\frac{1}{2}}\mathbf{j}_{\frac{n}{2}}\\
-D_2^{\frac{1}{2}}\mathbf{j}_{\frac{n}{2}}
\end{matrix} \right),
\end{equation*}
one can easily verify that $\mathbf{x}_1$ and $\mathbf{x}_2$ are the eigenvectors of $G$ corresponding to the $\mathcal{L}$-eigenvalues $0$ and $2$, respectively. Let $f(x)=(x-\alpha)(x-(2-\alpha))$. By using the spectral decomposition of $f(\mathcal{L})$, we obtain 
\begin{equation*}
f(\mathcal{L})=(\mathcal{L}-\alpha I)(\mathcal{L}-(2-\alpha)I)=\alpha(2-\alpha)(\mathbf{x}_1\mathbf{x}_1^T+\mathbf{x}_2\mathbf{x}_2^T),
\end{equation*}
that is,
\begin{equation}\label{equ-4-1}
\bigg(D^{-\frac{1}{2}}AD^{-\frac{1}{2}}\bigg)^2-(1-\alpha)^2I=\frac{\alpha(2-\alpha)}{m}
\left(\begin{matrix}
D_1^{\frac{1}{2}}J_{\frac{n}{2}}D_1^{\frac{1}{2}}& 0\\
0&D_2^{\frac{1}{2}}J_{\frac{n}{2}}D_2^{\frac{1}{2}}\\
\end{matrix}\right).
\end{equation}
By considering the $(u,u)$-entry and $(u,v)$-entry ($u,v\in V_i$ and $u\neq v$) at both sides of (\ref{equ-4-1}), we have
\begin{equation}\label{equ-4-2}
\sum_{w\sim u}\frac{1}{d_w}=(1-\alpha)^2d_u+\frac{\alpha(2-\alpha)}{m}d_u^2,
\end{equation}
and 
\begin{equation}\label{equ-4-3}
\sum_{\begin{smallmatrix}w\sim u\\
w\sim v\end{smallmatrix}}\frac{1}{d_w}=\frac{\alpha(2-\alpha)}{m}d_ud_v~\mbox{ for $u,v\in V_i$, where $i=1,2$}.
\end{equation}
Suppose that $u_0$ is a vertex of $G$ with degree $1$. Assume that $u_0\in V_1$, and  let $v_0\in V_2$ be the unique neighbor of $u_0$. For any $u\in V_1\setminus\{u_0\}$, from (\ref{equ-4-3}) we know that 
\begin{equation}\label{equ-4-7}
\sum_{\begin{smallmatrix}w\sim u\\
w\sim u_0\end{smallmatrix}}\frac{1}{d_w}=\frac{\alpha(2-\alpha)}{m}d_ud_{u_0}>0,
\end{equation}
implying that  $u\sim v_0$, and $v_0$ is the unique common neighbor of  $u$ and $u_0$  due to $d_{u_0}=1$. Thus $v_0$ is adjacent to all vertices in $V_1$ by the arbitrariness of $u$, that is, $d_{v_0}=\frac{n}{2}$, and furthermore,  from (\ref{equ-4-3}) we have 
\begin{equation}\label{equ-4-8}
\frac{2}{n}=\frac{1}{d_{v_0}}=\sum_{\begin{smallmatrix}w\sim u\\
w\sim u_0\end{smallmatrix}}\frac{1}{d_w}=\frac{\alpha(2-\alpha)}{m}d_u=\frac{1-(1-\alpha)^2}{m} d_u~\mbox{for any $u\in V_1\setminus\{u_0\}$}.
\end{equation}
 In addition, putting $u=u_0$ in (\ref{equ-4-2}) and noting that $v_0$ is the unique neighbor of $u_0$, we obtain
\begin{equation}\label{equ-4-9}
\frac{2}{n}=(1-\alpha)^2+\frac{\alpha(2-\alpha)}{m}, \mbox{ i.e., } (1-\alpha)^2=\frac{2m-n}{n(m-1)}.
\end{equation}
Combining (\ref{equ-4-8}) and (\ref{equ-4-9}), we deduce that  $d_u=\frac{2m-2}{n-2}$ for any $u\in V_1\setminus\{u_0\}$. For any $v\in V_2\setminus\{v_0\}$, we see that all neighbors of $v$ share the same degree $\frac{2m-2}{n-2}$. Then from (\ref{equ-4-2}) we get
\begin{equation}\label{equ-4-10}
\frac{n-2}{2m-2}\cdot d_v=\sum_{w\sim v}\frac{1}{d_w}=(1-\alpha)^2d_v+\frac{1-(1-\alpha)^2}{m}d_v^2.
\end{equation}
Combining (\ref{equ-4-9}) and (\ref{equ-4-10}), one can easily deduce that $d_v=\frac{n^2-4m}{2n-4}$. Since the sum of all degrees of vertices in $V_2$ is equal to  the number of edges of $G$, we have $\frac{n}{2}+\frac{n^2-4m}{2n-4}\cdot(\frac{n}{2}-1)=m$, which implies that $m=\frac{n^2+2n}{8}$. Therefore, we get $\alpha=1-\sqrt{\frac{2}{n+4}}$ by (\ref{equ-4-9}), and $d_u=\frac{n+4}{4}$ for any $u\in V_1\setminus\{u_0\}$, $d_v=\frac{n}{4}$ for any $v\in V_2\setminus\{v_0\}$. If $n=4$, it is obvious that $G=P_4$, which has $\mathcal{L}$-spectrum $\mathrm{Spec}_{\mathcal{L}}(P_4)=\{[2],[1.5],[0.5],[0]\}$, as required. If $n>4$, then $n\geq 6$ due to $n$ is even, and so $|V_1\setminus{u_0}|=|V_2\setminus{v_0}|\geq 2$.  Then, for any two vertices $u_1,u_2\in V_1\setminus\{u_0\}$, from (\ref{equ-4-3}) we obtain
\begin{equation*}
\frac{2}{n}+\big|[N_{G}(u_1)\cap N_G(u_2)]\setminus\{v_0\}\big|\cdot\frac{4}{n}=\sum_{\begin{smallmatrix}w\sim u_1\\
w\sim u_2\end{smallmatrix}}\frac{1}{d_w}=\frac{\alpha(2-\alpha)}{m}d_{u_1}d_{u_2}=\frac{n+4}{2n},
\end{equation*}
which gives that $\big|[N_{G}(u_1)\cap N_G(u_2)]\setminus\{v_0\}\big|=\frac{n}{8}$. Similarly, for  $v_1,v_2\in V_2\setminus\{v_0\}$, we have
\begin{equation*}
\big|N_{G}(v_1)\cap N_G(v_2)\big|\cdot\frac{4}{n+4}=\sum_{\begin{smallmatrix}w\sim v_1\\
w\sim v_2\end{smallmatrix}}\frac{1}{d_w}=\frac{\alpha(2-\alpha)}{m}d_{v_1}d_{v_2}=\frac{n}{2(n+4)},
\end{equation*}
so $\big|N_{G}(v_1)\cap N_G(v_2)\big|=\frac{n}{8}$. Set $V_1'=V_1\setminus\{u_0\}$, $V_2'=V_2\setminus\{v_0\}$ and $G'=G[V_1'\cup V_2']$, the induced subgraph of $G$ on $V_1'\cup V_2'$. Then $G$ is just the graph obtained from $G'\cup K_2$ by joining one vertex of $K_2$ to all the vertices in $V_1'$ of $G'$. By above arguments, we see that  $G'$   is a $\frac{n}{4}$-regular bipartite graph (with the bipartition $V(G')=V_1'\cup V_2'$) on $n-2$ vertices in which each pair of vertices in $V_1'$ (resp. $V_2'$) have $\frac{n}{8}$ common neighbors in $V_2'$ (resp. $V_1'$).  Therefore, we claim that $G'$ is the incidence graph of a symmetric BIBD with parameters $(4t-1,2t,t)$ if we put $n=8t$. Clearly, such a symmetric BIBD is the complement of a symmetric BIBD with parameters $(4t-1,2t-1,t-1)$, which is known as the Hadamard design of dimension $t$. 
 
 Conversely, if $G'=K_2$, by the assumption, we have $G=P_4$, which has exactly four distinct $\mathcal{L}$-eigenvalues.  Now assume that $G'$  is the incidence graph (with the bipartition $V(G')=V_1'\cup V_2'$)  of the complement of a Hadamard design of dimension $t$. In other words, $G'$ is the incidence graph of a symmetric BIBD  with parameters $(4t-1,2t,t)$. Recall that $G$ is the graph obtained from $G'\cup K_2$ ($V(K_2)=\{u_0,v_0\}$) by joining the vertex $v_0$ of $K_2$ to all the vertices in $V_1'$  of $G'$.  We will show that $G$ has exactly four distinct $\mathcal{L}$-eigenvalues. 
 Suppose $A(G')=\left(\begin{matrix}0&B'\\B^{'T}&0\end{matrix}\right)$. Then the $\mathcal{L}$-matrix of $G$ can be written as 
 \begin{equation*}
 \mathcal{L}=I_{8t}-
\left(\begin{matrix}
0&\sqrt{1/4t}&0&0\\
\sqrt{1/4t}&0&\sqrt{1/(8t^2+4t)}\mathbf{j}_{4t-1}^T&0\\
0&\sqrt{1/(8t^2+4t)}\mathbf{j}_{4t-1}&0&\sqrt{1/(4t^2+2t)}B'\\
0&0&\sqrt{1/(4t^2+2t)}B'^T&0
\end{matrix}\right)
\begin{matrix}
u_0\\
\vspace{0.05cm}
v_0\\
\vspace{0.05cm}
V_1'\\
\vspace{0.05cm}
V_2'
\end{matrix}.
\end{equation*}
By the arguments at the beginning  of  this section and (\ref{equ-4-4}), we know that $G'$ is a $2t$-regular bipartite graph on $8t-2$ vertices with  adjacency spectrum  
\begin{equation*}
\mathrm{Spec}_{A}(G')=\Big\{[2t]^1,\Big[\sqrt{t}\Big]^{4t-2},\Big[-\sqrt{t}\Big]^{4t-2},[-2t]^1\Big\}.
\end{equation*}
Note that the vectors $\mathbf{y}_0=\mathbf{j}_{8t-2}=(\mathbf{j}_{4t-1}^T,\mathbf{j}_{4t-1}^T)^T$ and $\mathbf{y}_0'=(\mathbf{j}_{4t-1}^T,-\mathbf{j}_{4t-1}^T)^T$ are the eigenvectors of $G'$ with respect to the (adjacency) eigenvalues $2t$ and $-2t$, respectively. Suppose that $\mathbf{y_i}=(\mathbf{z}_i^T,\mathbf{w}_i^T)^T$ and $\mathbf{y_i}'=(\mathbf{z}_i^T,-\mathbf{w}_i^T)^T$ ($1\leq i\leq 4t-2$) are all the orthonormal eigenvectors of $G'$ with respect to the (adjacency) eigenvalues $\sqrt{t}$ and $-\sqrt{t}$, respectively. Then $B'\mathbf{w}_i=\sqrt{t}\mathbf{z}_i$ and $B'^T\mathbf{z}_i=\sqrt{t}\mathbf{w}_i$ for $1\leq i\leq 4t-2$. Also, for each $i$, we have $\mathbf{z}_i^T\mathbf{j}_{4t-1}=0$ and $\mathbf{w}_i^T\mathbf{j}_{4t-1}=0$ due to $\mathbf{y}_i^T\mathbf{y}_0=0$ and $\mathbf{y}_i^T\mathbf{y}_0'=0$. 

Taking $\mathbf{x}_0=D^{\frac{1}{2}}\mathbf{j}_{8t}=(1,2\sqrt{t},\sqrt{2t+1}\mathbf{j}_{4t-1}^T,\sqrt{2t}\mathbf{j}_{4t-1}^T)^T$ and $\mathbf{x}_0'=(1,-2\sqrt{t},\sqrt{2t+1}\mathbf{j}_{4t-1}^T,$ $-\sqrt{2t}\mathbf{j}_{4t-1}^T)^T$, one can easily verify that $\mathcal{L}\mathbf{x}_0=0\cdot \mathbf{x}_0$ and $\mathcal{L}\mathbf{x}_0'=2\cdot \mathbf{x}_0'$ due to $B'\mathbf{j}_{4t-1}=B'^T\mathbf{j}_{4t-1}=2t\cdot\mathbf{j}_{4t-1}$. Furthermore, suppose $\mathbf{x}_i=(0,0,\mathbf{z}_i^T,\mathbf{w}_i^T)^T$, $\mathbf{x}_i'=(0,0,\mathbf{z}_i^T,-\mathbf{w}_i^T)^T$ for $1\leq i\leq 4t-2$, and $\mathbf{x}_{4t-1}=(1,\sqrt{\frac{2t}{2t+1}},-\frac{1}{(4t-1)\sqrt{2t+1}}\mathbf{j}_{4t-1}^T,-\frac{2\sqrt{t}}{(4t-1)\sqrt{2t+1}}\mathbf{j}_{4t-1}^T)^T$, $\mathbf{x}_{4t-1}'=(1,-\sqrt{\frac{2t}{2t+1}},$ $-\frac{1}{(4t-1)\sqrt{2t+1}}\mathbf{j}_{4t-1}^T,\frac{2\sqrt{t}}{(4t-1)\sqrt{2t+1}}\mathbf{j}_{4t-1}^T)^T$.  Since $\mathbf{z}_i^T\mathbf{j}_{4t-1}=0$, $\mathbf{w}_i^T\mathbf{j}_{4t-1}=0$,   and $B'\mathbf{j}_{4t-1}=B'^T\mathbf{j}_{4t-1}=2t\cdot\mathbf{j}_{4t-1}$, one can also verify that $\mathcal{L}\mathbf{x}_i=(1-\sqrt{1/(4t+2)})\mathbf{x}_i$ and $\mathcal{L}\mathbf{x}_i'=(1+\sqrt{1/(4t+2)})\mathbf{x}_i'$ holds for each $1\leq i\leq 4t-1$.  Since $\mathbf{x}_1,\mathbf{x}_1',\ldots,\mathbf{x}_{4t-1},\mathbf{x}_{4t-1}'$ are pairwise orthogonal, we conclude that both $1-\sqrt{1/(4t+2)}$ and $1+\sqrt{1/(4t+2)}$ are the $\mathcal{L}$-eigenvalues of $G$ with multiplicities at least $4t-1$. As $(4t-1)\cdot2+2=8t$, which equals to order of $G$, we have 
\begin{equation*}
\mathrm{Spec}_{\mathcal{L}}(G)=\Big\{[2]^1,\Big[1+\sqrt{1/(4t+2)}\Big]^{4t-1},\Big[1-\sqrt{1/(4t+2)}\Big]^{4t-1},[0]^1\Big\},
\end{equation*}
and our result follows.
\end{proof}

Let  $\mathcal{G}$ denote the set of connected bipartite graphs with at least one vertex of degree $1$ having four distinct $\mathcal{L}$-eigenvalues. According to Theorem \ref{thm-4-1}, each  graph (except $P_4$) in $\mathcal{G}$ is of order  $n=8t$ for some positive integer $t$ and corresponds to a Hadamard  design of dimension $t$, or equivalently, a Hadamard matrix of order $4t$. In the following, we list some examples on constructing Hadamard matrices.

The \emph{Kronecker product} $M\otimes N$ of matrices $M=(m_{ij})_{a\times b}$ and $N=(n_{ij})_{c\times d}$ is the $ac\times bd$ matrix obtained from $M$ by replacing each element $m_{ij}$ with the block $m_{ij}N$. 
\begin{example}(Sylvester's Construction, see \cite{Sylvester}) 
\emph{Assume that $H_1$ and $H_2$ are two Hadamard matrices of order $m$ and $n$, respectively.  Then $H_1H_1^T=mI_m$ and $H_2H_2^T=nI_n$. Therefore, $(H_1\otimes H_2)(H_1\otimes H_2)^T=(H_1\otimes H_2)(H_1^T\otimes H_2^T)=(H_1H_1^T)\otimes(H_2H_2^T)=mI_m\otimes nI_n=mnI_{mn}$, which implies that $H_1\otimes H_2$ is a Hadamard matrix of order $mn$. Let $H$ be the Hadamard matrix of order $2$ given by}
\begin{equation*}
H=\left(
\begin{matrix}
1&1\\
1&-1\\
\end{matrix}
\right).
\end{equation*}
\emph{Putting $H_1=H$ and $H_i=H\otimes H_{i-1}$ for $i\geq 2$. We see that $H_i$ ($i\geq 2$) is also a Hadamard matrix, which has  order $2^i=4\cdot 2^{i-2}$. Therefore, by Theorem \ref{thm-4-1}, there exists a connected bipartite graph belonging to $\mathcal{G}$ of order $8\cdot 2^{i-2}=2^{i+1}$ for each $i\geq 2$.}
\end{example}

\begin{example}(Paley's Constructions, see \cite{Paley}, or \cite{Hedayat}, Theorems 3.2--3.3) 
\emph{Firstly, let $p^\alpha$ be a prime power such that $p^\alpha\equiv 3~(\mathrm{mod}~4)$, and let  $a_0,a_1,\ldots,$ $a_{p^\alpha-1}$  be all the elements of the finite field $GF(p^\alpha)$. Suppose that $C=(c_{ij})$ is the matrix of order $p^\alpha$ defined as follows:
\begin{equation*}
c_{ij}=\left\{
\begin{array}{ll}
0&\mbox{if $a_i=a_j$,}\\
1&\mbox{if $a_i-a_j=a^2$ for some $a\in GF(p^\alpha)\setminus\{0\}$,}\\
-1&\mbox{if $a_i-a_j$ is not a square in $GF(p^\alpha)$.}\\
\end{array}
\right.
\end{equation*}
Putting 
\begin{equation*}
S_1=\left(
\begin{matrix}
0&\mathbf{j}^T\\
-\mathbf{j}&C\\
\end{matrix}
\right)~~\mbox{and}~~H_1=I+S_1.
\end{equation*}
Then $H_1$ is  a Hadamard matrix of order $p^\alpha+1$ ($=4t_1$). Next, let $p^\beta$ be a prime power such that $p^\beta\equiv 1~(\mathrm{mod}~4)$. A matrix $C$ could be constructed as above. Putting 
\begin{equation*}
S_2=\left(
\begin{matrix}
0&\mathbf{j}^T\\
-\mathbf{j}&C\\
\end{matrix}
\right)~~\mbox{and}~~H_2=S_2\otimes \left(\begin{matrix}1&1\\1&-1\\\end{matrix}\right)+I\otimes \left(\begin{matrix}1&-1\\-1&-1\\\end{matrix}\right),
\end{equation*} 
$H_2$ is  a Hadamard matrix of order $2(p^\beta+1)$ ($=4t_2$). 
Therefore, by Theorem \ref{thm-4-1}, there exists a connected bipartite graph belonging to $\mathcal{G}$ of order $2(p^\alpha +1)=8t_1$ when $p^\alpha\equiv 3~(\mathrm{mod}~4)$, and of order $4(p^\beta+1)=8t_2$ when $p^\beta\equiv 1~(\mathrm{mod}~4)$.}
\end{example}

In addition, Wallis in \cite{Wallis} proved that, if $q$ is an odd natural number, there exists a Hadamard matrix  of order $2^sq$ for each natural number $s\geq [2\log_2(q-3)]$. Therefore, given any odd natural number  $q$, there exists a connected bipartite graph belonging to $\mathcal{G}$ of order $2^{s+1}q$ for each $s\geq [2\log_2(q-3)]$. For more techniques on the construction of Hadamard matrices or Hadamard designs, we refer the reader to \cite{Hedayat,Seberry,Craigen}.

\section{Unicyclic  graphs with  three or four distinct $\mathcal{L}$-eigenvalues}
\begin{figure}
\unitlength 2mm 
\linethickness{0.4pt}
\ifx\plotpoint\undefined\newsavebox{\plotpoint}\fi 
\begin{picture}(78,45.3)(0,0)
\put(11,35){\line(0,1){0}}
\put(14,41){\line(-2,-3){4}}
\put(10,35){\line(1,0){8}}
\put(18,35){\line(-2,3){4}}
\put(24,35){\line(0,1){0}}
\put(27,41){\line(-2,-3){4}}
\put(23,35){\line(1,0){8}}
\put(31,35){\line(-2,3){4}}
\put(39,35){\line(2,3){4}}
\put(43,41){\line(2,-3){4}}
\put(47,35){\line(-1,0){8}}
\put(58,35){\line(1,0){8}}
\put(36,39){\line(3,-4){3}}
\put(39,35){\line(-3,-4){3}}
\multiput(36,33)(.05,.03333333){60}{\line(1,0){.05}}
\multiput(39,35)(-.05,.03333333){60}{\line(-1,0){.05}}
\put(47,35){\line(3,4){3}}
\multiput(47,35)(.05,.03333333){60}{\line(1,0){.05}}
\put(47,35){\line(3,-4){3}}
\multiput(47,35)(.05,-.03333333){60}{\line(1,0){.05}}
\put(58,35){\line(-3,4){3}}
\multiput(55,37)(.05,-.03333333){60}{\line(1,0){.05}}
\multiput(58,35)(-.05,-.03333333){60}{\line(-1,0){.05}}
\put(58,35){\line(-3,-4){3}}
\multiput(66,35)(.05,.03333333){60}{\line(1,0){.05}}
\put(66,35){\line(3,4){3}}
\put(66,35){\line(3,-4){3}}
\multiput(66,35)(.05,-.03333333){60}{\line(1,0){.05}}
\put(27,41){\line(-4,3){4}}
\put(27,41){\line(4,3){4}}
\multiput(27,41)(.03333333,.05){60}{\line(0,1){.05}}
\multiput(27,41)(-.03333333,.05){60}{\line(0,1){.05}}
\put(14,41){\circle*{1.2}}
\put(10,35){\circle*{1.2}}
\put(18,35){\circle*{1.2}}
\put(23,44){\circle*{1.2}}
\put(25,44){\circle*{1.2}}
\put(29,44){\circle*{1.2}}
\put(31,44){\circle*{1.2}}
\put(27,41){\circle*{1.2}}
\put(23,35){\circle*{1.2}}
\put(31,35){\circle*{1.2}}
\put(36,39){\circle*{1.2}}
\put(36,37){\circle*{1.2}}
\put(43,41){\circle*{1.2}}
\put(39,35){\circle*{1.2}}
\put(47,35){\circle*{1.2}}
\put(50,39){\circle*{1.2}}
\put(50,37){\circle*{1.2}}
\put(50,33){\circle*{1.2}}
\put(50,31){\circle*{1.2}}
\put(36,31){\circle*{1.2}}
\put(36,33){\circle*{1.2}}
\put(58,35){\circle*{1.2}}
\put(66,35){\circle*{1.2}}
\put(55,39){\circle*{1.2}}
\put(55,37){\circle*{1.2}}
\put(55,33){\circle*{1.2}}
\put(55,31){\circle*{1.2}}
\put(69,39){\circle*{1.2}}
\put(69,37){\circle*{1.2}}
\put(69,33){\circle*{1.2}}
\put(69,31){\circle*{1.2}}
\put(27,44){\makebox(0,0)[cc]{$\cdots$}}
\put(36,35.5){\makebox(0,0)[cc]{$\vdots$}}
\put(50,35.5){\makebox(0,0)[cc]{$\vdots$}}
\put(55,35.5){\makebox(0,0)[cc]{$\vdots$}}
\put(69,35.5){\makebox(0,0)[cc]{$\vdots$}}
\put(14,29.5){\makebox(0,0)[cc]{\scriptsize$U_1$}}
\put(27,29.5){\makebox(0,0)[cc]{\scriptsize$U_2(a)$}}
\put(43,29.5){\makebox(0,0)[cc]{\scriptsize$U_3(a,b)$}}
\put(23,45.3){\makebox(0,0)[cc]{\scriptsize$u_1$}}
\put(25,45.3){\makebox(0,0)[cc]{\scriptsize$u_2$}}
\put(29,45.3){\makebox(0,0)[cc]{\scriptsize$u_{a-1}$}}
\put(31,45.3){\makebox(0,0)[cc]{\scriptsize$u_a$}}
\put(34.5,39){\makebox(0,0)[cc]{\scriptsize$u_1$}}
\put(34.5,37){\makebox(0,0)[cc]{\scriptsize$u_2$}}
\put(34.2,33){\makebox(0,0)[cc]{\scriptsize$u_{a-1}$}}
\put(34.5,31){\makebox(0,0)[cc]{\scriptsize$u_a$}}
\put(51.5,39){\makebox(0,0)[cc]{\scriptsize$v_1$}}
\put(51.5,37){\makebox(0,0)[cc]{\scriptsize$v_2$}}
\put(51.4,33.7){\makebox(0,0)[cc]{\scriptsize$v_{b-1}$}}
\put(51.5,31){\makebox(0,0)[cc]{\scriptsize$v_{b}$}}
\put(53.5,39){\makebox(0,0)[cc]{\scriptsize$u_1$}}
\put(53.5,37){\makebox(0,0)[cc]{\scriptsize$u_2$}}
\put(53.6,33.7){\makebox(0,0)[cc]{\scriptsize$u_{a-1}$}}
\put(53.5,31){\makebox(0,0)[cc]{\scriptsize$u_a$}}
\put(70.5,39){\makebox(0,0)[cc]{\scriptsize$v_1$}}
\put(70.5,37){\makebox(0,0)[cc]{\scriptsize$v_2$}}
\put(70.7,33){\makebox(0,0)[cc]{\scriptsize$v_{b-1}$}}
\put(70.5,31){\makebox(0,0)[cc]{\scriptsize$v_b$}}
\put(9,21){\line(1,0){8}}
\put(17,21){\line(-2,3){4}}
\put(13,27){\line(-2,-3){4}}
\put(17,21){\line(1,0){4}}
\put(21,21){\line(3,4){3}}
\put(21,21){\line(3,-4){3}}
\multiput(21,21)(.05,-.03333333){60}{\line(1,0){.05}}
\multiput(21,21)(.05,.03333333){60}{\line(1,0){.05}}
\put(29,21){\line(1,0){8}}
\put(37,21){\line(-2,3){4}}
\put(33,27){\line(-2,-3){4}}
\put(37,21){\line(1,0){4}}
\put(41,21){\line(3,4){3}}
\put(41,21){\line(3,-4){3}}
\multiput(41,21)(.05,-.03333333){60}{\line(1,0){.05}}
\multiput(41,21)(.05,.03333333){60}{\line(1,0){.05}}
\put(37,21){\line(-4,-3){4}}
\put(37,21){\line(4,-3){4}}
\multiput(37,21)(-.03333333,-.05){60}{\line(0,-1){.05}}
\multiput(37,21)(.03333333,-.05){60}{\line(0,-1){.05}}
\put(49,27){\line(0,-1){6}}
\put(49,21){\line(1,0){6}}
\put(55,21){\line(0,1){6}}
\put(55,27){\line(-1,0){6}}
\put(60,27){\line(0,-1){6}}
\put(60,21){\line(1,0){6}}
\put(66,21){\line(0,1){6}}
\put(66,27){\line(-1,0){6}}
\put(66,21){\line(3,-4){3}}
\multiput(66,21)(.05,.03333333){60}{\line(1,0){.05}}
\multiput(66,21)(.05,-.03333333){60}{\line(1,0){.05}}
\put(66,21){\line(3,4){3}}
\put(13,27){\circle*{1.2}}
\put(9,21){\circle*{1.2}}
\put(17,21){\circle*{1.2}}
\put(24,25){\circle*{1.2}}
\put(24,23){\circle*{1.2}}
\put(24,19){\circle*{1.2}}
\put(24,17){\circle*{1.2}}
\put(33,27){\circle*{1.2}}
\put(29,21){\circle*{1.2}}
\put(37,21){\circle*{1.2}}
\put(33,18){\circle*{1.2}}
\put(35,18){\circle*{1.2}}
\put(39,18){\circle*{1.2}}
\put(41,18){\circle*{1.2}}
\put(44,25){\circle*{1.2}}
\put(44,23){\circle*{1.2}}
\put(44,19){\circle*{1.2}}
\put(44,17){\circle*{1.2}}
\put(49,27){\circle*{1.2}}
\put(55,27){\circle*{1.2}}
\put(49,21){\circle*{1.2}}
\put(55,21){\circle*{1.2}}
\put(60,27){\circle*{1.2}}
\put(60,21){\circle*{1.2}}
\put(66,27){\circle*{1.2}}
\put(66,21){\circle*{1.2}}
\put(69,25){\circle*{1.2}}
\put(69,23){\circle*{1.2}}
\put(69,19){\circle*{1.2}}
\put(69,17){\circle*{1.2}}
\put(24,21.5){\makebox(0,0)[cc]{$\vdots$}}
\put(37,18){\makebox(0,0)[cc]{$\cdots$}}
\put(44,21.5){\makebox(0,0)[cc]{$\vdots$}}
\put(69,21.5){\makebox(0,0)[cc]{$\vdots$}}
\put(41,21){\circle*{1.2}}
\put(21,21){\circle*{1.2}}
\put(62,29.5){\makebox(0,0)[cc]{\scriptsize$U_4(a,b,c)$}}
\put(16,15.5){\makebox(0,0)[cc]{\scriptsize$U_5(a)$}}
\put(37,15.5){\makebox(0,0)[cc]{\scriptsize$U_6(a,b)$}}
\put(52,15.5){\makebox(0,0)[cc]{\scriptsize$U_7$}}
\put(63,15.5){\makebox(0,0)[cc]{\scriptsize$U_8(a)$}}
\put(25.5,25){\makebox(0,0)[cc]{\scriptsize$u_1$}}
\put(25.5,23){\makebox(0,0)[cc]{\scriptsize$u_2$}}
\put(25.7,19){\makebox(0,0)[cc]{\scriptsize$u_{a-1}$}}
\put(25.5,17){\makebox(0,0)[cc]{\scriptsize$u_a$}}
\put(45.5,25){\makebox(0,0)[cc]{\scriptsize$u_1$}}
\put(45.5,23){\makebox(0,0)[cc]{\scriptsize$u_2$}}
\put(45.7,19){\makebox(0,0)[cc]{\scriptsize$u_{a-1}$}}
\put(45.5,17){\makebox(0,0)[cc]{\scriptsize$u_a$}}
\put(33,16.8){\makebox(0,0)[cc]{\scriptsize$v_{1}$}}
\put(35,16.8){\makebox(0,0)[cc]{\scriptsize$v_{2}$}}
\put(39,16.8){\makebox(0,0)[cc]{\scriptsize$v_{b-1}$}}
\put(41,16.8){\makebox(0,0)[cc]{\scriptsize$v_{b}$}}
\put(70.5,25){\makebox(0,0)[cc]{\scriptsize$u_{1}$}}
\put(70.5,23){\makebox(0,0)[cc]{\scriptsize$u_{2}$}}
\put(70.7,19){\makebox(0,0)[cc]{\scriptsize$u_{a-1}$}}
\put(70.5,17){\makebox(0,0)[cc]{\scriptsize$u_{a}$}}
\put(28.5,40.8){\makebox(0,0)[cc]{\scriptsize$u_0$}}
\put(23,33.5){\makebox(0,0)[cc]{\scriptsize$v_0$}}
\put(31,33.5){\makebox(0,0)[cc]{\scriptsize$w_0$}}
\put(44.5,40.8){\makebox(0,0)[cc]{\scriptsize$w_0$}}
\put(39,33.5){\makebox(0,0)[cc]{\scriptsize$u_0$}}
\put(47,33.5){\makebox(0,0)[cc]{\scriptsize$v_0$}}
\put(58,33.5){\makebox(0,0)[cc]{\scriptsize$u_0$}}
\put(66,33.5){\makebox(0,0)[cc]{\scriptsize$v_0$}}
\put(21,19.5){\makebox(0,0)[cc]{\scriptsize$u_0$}}
\put(17,19.5){\makebox(0,0)[cc]{\scriptsize$v_0$}}
\put(9,19.5){\makebox(0,0)[cc]{\scriptsize$w_0$}}
\put(41,22.5){\makebox(0,0)[cc]{\scriptsize$u_0$}}
\put(37,22.5){\makebox(0,0)[cc]{\scriptsize$v_0$}}
\put(29,19.5){\makebox(0,0)[cc]{\scriptsize$w_0$}}
\put(66,19.5){\makebox(0,0)[cc]{\scriptsize$u_0$}}
\put(60,19.5){\makebox(0,0)[cc]{\scriptsize$v_0$}}
\put(61,26){\makebox(0,0)[cc]{\scriptsize$w_0$}}
\put(1.7,11){\line(3,-4){3}}
\put(4.7,7){\line(-3,-4){3}}
\multiput(1.7,5)(.05,.03333333){60}{\line(1,0){.05}}
\multiput(4.7,7)(-.05,.03333333){60}{\line(-1,0){.05}}
\put(4.7,13){\line(0,-1){6}}
\put(4.7,13){\line(1,0){6}}
\put(10.7,13){\line(0,-1){6}}
\put(10.7,7){\line(-1,0){6}}
\put(10.7,7){\line(3,4){3}}
\multiput(10.7,7)(.05,.03333333){60}{\line(1,0){.05}}
\put(10.7,7){\line(3,-4){3}}
\multiput(10.7,7)(.05,-.03333333){60}{\line(1,0){.05}}
\put(18.7,7){\line(1,0){6}}
\put(24.7,7){\line(0,1){4}}
\multiput(24.7,11)(-.05,.03333333){60}{\line(-1,0){.05}}
\multiput(21.7,13)(-.05,-.03333333){60}{\line(-1,0){.05}}
\put(18.7,11){\line(0,-1){4}}
\put(35.7,7){\line(-1,0){6}}
\put(35.7,7){\line(3,4){3}}
\multiput(35.7,7)(.05,.03333333){60}{\line(1,0){.05}}
\put(35.7,7){\line(3,-4){3}}
\multiput(35.7,7)(.05,-.03333333){60}{\line(1,0){.05}}
\put(29.7,7){\line(0,1){4}}
\multiput(29.7,11)(.05,.03333333){60}{\line(1,0){.05}}
\multiput(32.7,13)(.05,-.03333333){60}{\line(1,0){.05}}
\put(35.7,11){\line(0,-1){4}}
\put(4.7,13){\circle*{1.2}}
\put(10.7,13){\circle*{1.2}}
\put(4.7,7){\circle*{1.2}}
\put(10.7,7){\circle*{1.2}}
\put(1.7,11){\circle*{1.2}}
\put(1.7,9){\circle*{1.2}}
\put(1.7,5){\circle*{1.2}}
\put(1.7,3){\circle*{1.2}}
\put(13.7,11){\circle*{1.2}}
\put(13.7,9){\circle*{1.2}}
\put(13.7,5){\circle*{1.2}}
\put(13.7,3){\circle*{1.2}}
\put(21.7,13){\circle*{1.2}}
\put(18.7,11){\circle*{1.2}}
\put(18.7,7){\circle*{1.2}}
\put(24.7,7){\circle*{1.2}}
\put(24.7,11){\circle*{1.2}}
\put(32.7,13){\circle*{1.2}}
\put(29.7,11){\circle*{1.2}}
\put(29.7,7){\circle*{1.2}}
\put(35.7,11){\circle*{1.2}}
\put(35.7,7){\circle*{1.2}}
\put(38.7,11){\circle*{1.2}}
\put(38.7,9){\circle*{1.2}}
\put(38.7,5){\circle*{1.2}}
\put(38.7,3){\circle*{1.2}}
\put(1.7,7.5){\makebox(0,0)[cc]{$\vdots$}}
\put(13.7,7.5){\makebox(0,0)[cc]{$\vdots$}}
\put(38.7,7.5){\makebox(0,0)[cc]{$\vdots$}}
\put(7.7,0){\makebox(0,0)[cc]{\scriptsize$U_9(a,b)$}}
\put(21.7,0){\makebox(0,0)[cc]{\scriptsize$U_{10}$}}
\put(33.7,0){\makebox(0,0)[cc]{\scriptsize$U_{11}(a)$}}
\put(.2,11){\makebox(0,0)[cc]{\scriptsize$u_{1}$}}
\put(.2,9){\makebox(0,0)[cc]{\scriptsize$u_{2}$}}
\put(0,5){\makebox(0,0)[cc]{\scriptsize$u_{a-1}$}}
\put(.2,3){\makebox(0,0)[cc]{\scriptsize$u_{a}$}}
\put(15.2,11){\makebox(0,0)[cc]{\scriptsize$v_{1}$}}
\put(15.2,9){\makebox(0,0)[cc]{\scriptsize$v_{2}$}}
\put(15.4,5){\makebox(0,0)[cc]{\scriptsize$v_{b-1}$}}
\put(15.2,2){\makebox(0,0)[cc]{\scriptsize$v_{b}$}}
\put(43.7,11){\line(3,-4){3}}
\put(46.7,7){\line(-3,-4){3}}
\multiput(43.7,5)(.05,.03333333){60}{\line(1,0){.05}}
\multiput(46.7,7)(-.05,.03333333){60}{\line(-1,0){.05}}
\put(52.7,7){\line(-1,0){6}}
\put(52.7,7){\line(3,4){3}}
\multiput(52.7,7)(.05,.03333333){60}{\line(1,0){.05}}
\put(52.7,7){\line(3,-4){3}}
\multiput(52.7,7)(.05,-.03333333){60}{\line(1,0){.05}}
\put(60.7,11){\line(0,-1){4}}
\multiput(60.7,7)(.05,-.03333333){60}{\line(1,0){.05}}
\multiput(63.7,5)(.05,.03333333){60}{\line(1,0){.05}}
\put(66.7,7){\line(0,1){4}}
\multiput(66.7,11)(-.05,.03333333){60}{\line(-1,0){.05}}
\multiput(63.7,13)(-.05,-.03333333){60}{\line(-1,0){.05}}
\put(46.7,7){\line(0,1){4}}
\multiput(46.7,11)(.05,.03333333){60}{\line(1,0){.05}}
\multiput(49.7,13)(.05,-.03333333){60}{\line(1,0){.05}}
\put(52.7,11){\line(0,-1){4}}
\put(43.7,11){\circle*{1.2}}
\put(43.7,9){\circle*{1.2}}
\put(43.7,5){\circle*{1.2}}
\put(43.7,3){\circle*{1.2}}
\put(49.7,13){\circle*{1.2}}
\put(46.7,11){\circle*{1.2}}
\put(52.7,11){\circle*{1.2}}
\put(55.7,11){\circle*{1.2}}
\put(55.7,9){\circle*{1.2}}
\put(55.7,5){\circle*{1.2}}
\put(55.7,3){\circle*{1.2}}
\put(63.7,13){\circle*{1.2}}
\put(60.7,11){\circle*{1.2}}
\put(60.7,7){\circle*{1.2}}
\put(63.7,5){\circle*{1.2}}
\put(66.7,11){\circle*{1.2}}
\put(66.7,7){\circle*{1.2}}
\put(46.7,7){\circle*{1.2}}
\put(52.7,7){\circle*{1.2}}
\put(43.7,7.5){\makebox(0,0)[cc]{$\vdots$}}
\put(55.7,7.5){\makebox(0,0)[cc]{$\vdots$}}
\put(49.7,0){\makebox(0,0)[cc]{\scriptsize$U_{12}(a,b)$}}
\put(63.7,0){\makebox(0,0)[cc]{\scriptsize$U_{13}$}}
\put(40.2,11){\makebox(0,0)[cc]{\scriptsize$u_1$}}
\put(40.2,9){\makebox(0,0)[cc]{\scriptsize$u_2$}}
\put(40.1,5.7){\makebox(0,0)[cc]{\scriptsize$u_{a-1}$}}
\put(40.2,3){\makebox(0,0)[cc]{\scriptsize$u_a$}}
\put(42.2,11){\makebox(0,0)[cc]{\scriptsize$u_1$}}
\put(42.2,9){\makebox(0,0)[cc]{\scriptsize$u_2$}}
\put(42.3,5.7){\makebox(0,0)[cc]{\scriptsize$u_{a-1}$}}
\put(42.2,3){\makebox(0,0)[cc]{\scriptsize$u_a$}}
\put(57.2,11){\makebox(0,0)[cc]{\scriptsize$v_1$}}
\put(57.2,9){\makebox(0,0)[cc]{\scriptsize$v_2$}}
\put(57.4,5){\makebox(0,0)[cc]{\scriptsize$v_{b-1}$}}
\put(57.2,3){\makebox(0,0)[cc]{\scriptsize$v_{b}$}}
\put(4.7,5.5){\makebox(0,0)[cc]{\scriptsize$u_0$}}
\put(10.7,5.5){\makebox(0,0)[cc]{\scriptsize$v_0$}}
\put(5.7,12){\makebox(0,0)[cc]{\scriptsize$w_0$}}
\put(35.7,5.5){\makebox(0,0)[cc]{\scriptsize$u_0$}}
\put(29.7,5.5){\makebox(0,0)[cc]{\scriptsize$v_0$}}
\put(30.7,10){\makebox(0,0)[cc]{\scriptsize$w_0$}}
\put(46.7,5.5){\makebox(0,0)[cc]{\scriptsize$u_0$}}
\put(52.7,5.5){\makebox(0,0)[cc]{\scriptsize$v_0$}}
\put(47.7,10){\makebox(0,0)[cc]{\scriptsize$w_0$}}
\put(73.7,0){\makebox(0,0)[cc]{\scriptsize$U_{14}$}}
\put(71,11){\line(0,-1){4}}
\put(71,7){\line(0,-1){2}}
\put(77,5){\line(0,1){6}}
\multiput(77,11)(-.05,.03333333){60}{\line(-1,0){.05}}
\multiput(74,13)(-.05,-.03333333){60}{\line(-1,0){.05}}
\put(71,5){\line(1,0){6}}
\put(74,13){\circle*{1.2}}
\put(71,11){\circle*{1.2}}
\put(71,5){\circle*{1.2}}
\put(71,8){\circle*{1.2}}
\put(77,11){\circle*{1.2}}
\put(77,8){\circle*{1.2}}
\put(77,5){\circle*{1.2}}
\put(58,45.5){\makebox(0,0)[cc]{\scriptsize$w_1$}}
\put(60,45.5){\makebox(0,0)[cc]{\scriptsize$w_2$}}
\put(64,45.5){\makebox(0,0)[cc]{\scriptsize$w_{c-1}$}}
\put(66,45.5){\makebox(0,0)[cc]{\scriptsize$w_c$}}
\put(62,44){\makebox(0,0)[cc]{$\cdots$}}
\put(64,41){\makebox(0,0)[cc]{\scriptsize$w_0$}}
\put(58,35){\line(2,3){4}}
\put(62,41){\line(2,-3){4}}
\put(62,41){\line(-4,3){4}}
\multiput(60,44)(.03333333,-.05){60}{\line(0,-1){.05}}
\multiput(62,41)(.03333333,.05){60}{\line(0,1){.05}}
\put(62,41){\line(4,3){4}}
\put(62,41){\circle*{1.2}}
\put(58,44){\circle*{1.2}}
\put(60,44){\circle*{1.2}}
\put(64,44){\circle*{1.2}}
\put(66,44){\circle*{1.2}}
\end{picture}
\caption{\footnotesize Unicyclic graphs with diameter at most three ($a,b,c\geq 1$).}\label{fig-1}
\end{figure}
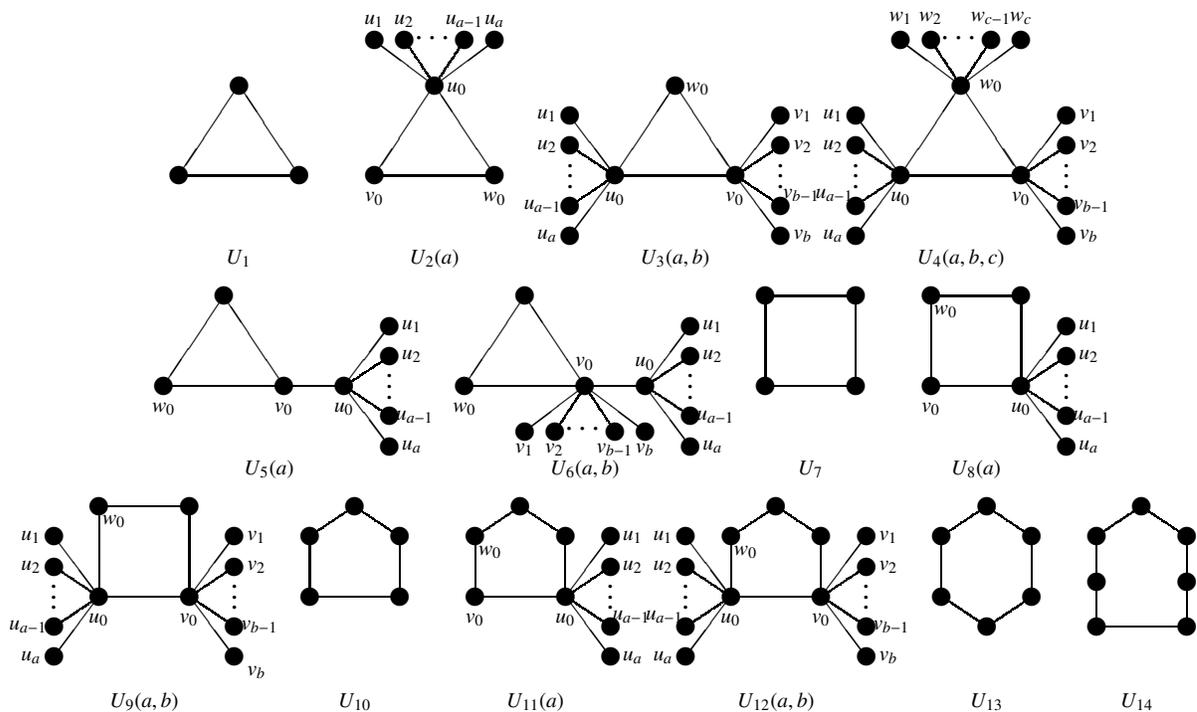

In \cite{Braga}, Braga et al. determine all trees with four or five distinct $\mathcal{L}$-eigenvalues.
Recall that a connected graph is called a unicyclic graph if it has $n$ vertices and $n$ edges. In this section, we completely characterize unicyclic  graphs with three or four distinct $\mathcal{L}$-eigenvalues.

The following lemma shows that the diameter of a connected graph is less than the number of distinct $\mathcal{L}$-eigenvalues.
\begin{lem}\label{lem-3-1} (See \cite{Chung}.)
Let $G$ be a connected graph with diameter $d$. If $G$ has $s$ distinct $\mathcal{L}$-eigenvalues, then $d\leq s-1$.
\end{lem}

According to Lemma \ref{lem-3-1},  connected graphs with at most four distinct $\mathcal{L}$-eigenvalues have diameter at most three. In Fig. \ref{fig-1}, we list all unicyclic graphs whose diameter are at most three for later use. The following theorem determines all unicyclic graphs with three distinct $\mathcal{L}$-eigenvalues.

\begin{thm}
Let $G$ be a unicyclic graph. Then $G$ has three distinct $\mathcal{L}$-eigenvalues if and only if $G=C_4$ or $G=C_5$.
\end{thm}
\begin{proof}
Assume that $G$ has three distinct $\mathcal{L}$-eigenvalues. By Lemma \ref{lem-3-1}, the diameter of $G$ is exactly two because it cannot be a complete graph, and so $G$ must be one of the following graphs: $U_2(a)$, $U_7=C_4$, $U_{10}=C_5$ (see Fig. \ref{fig-1}). First suppose that  $G=U_2(a)$. If $a=1$, then $\mathrm{Spec}_{\mathcal{L}}(U_2(1))=\{[1.7287], [1.5000],[0.7713],[0]\}$, and so $G\neq U_2(1)$.
 If $a\geq 2$, by Lemma \ref{lem-2-3}, $G$ has $1$ as its $\mathcal{L}$-eigenvalue. By Theorem \ref{thm-2-1},  $G$ must be a complete bipartite graph or a regular complete multipartite graph, both are impossible. Furthermore, by simple computation we know that $\mathrm{Spec}_{\mathcal{L}}(C_4)=\{[2], [1]^2,[0]\}$ and $\mathrm{Spec}_{\mathcal{L}}(C_5)=\{[1.8090]^2, [0.6910]^2,[0]\}$. Therefore, $G$ has three distinct $\mathcal{L}$-eigenvalues if and only if $G=C_4$ or $G=C_5$.
 \end{proof}

Let $G$ be a connected graph with four distinct $\mathcal{L}$-eigenvalues. By Lemma \ref{lem-2-2}, we may assume that $\mathrm{Spec}_{\mathcal{L}}(G)=\{[\alpha]^{m_1},[\beta]^{m_2},[\gamma]^{m_3},[0]^1\}$, where $2\geq \alpha> \beta> \gamma> 0$. According to Lemma \ref{lem-2-1}, we have
\begin{equation}\label{equ-8}
(\mathcal{L}-\alpha I)(\mathcal{L}-\beta I)(\mathcal{L}-\gamma I)=-\frac{\alpha\beta\gamma}{2m}D^{\frac{1}{2}}JD^{\frac{1}{2}},
\end{equation}
where $D$ is the diagonal degree matrix of $G$. By considering the $(u,u)$-entry  at both sides of (\ref{equ-8}), we obtain that
\begin{equation}\label{equ-9}
\sum_{\small\begin{smallmatrix}
u\sim v\\
v\sim w\\
w\sim u
\end{smallmatrix}}\frac{1}{d_vd_w}+(\alpha+\beta+\gamma-3)\sum_{v\sim u}\frac{1}{d_v}+(\alpha-1)(\beta-1)(\gamma-1)d_u=\frac{\alpha\beta\gamma}{2m}d_u^2.
\end{equation}

By using (\ref{equ-9}) and Lemma \ref{lem-2-3}, we now determine all unicyclic graphs with four distinct  $\mathcal{L}$-eigenvalues.

\begin{thm}
Let $G$ be a unicyclic graph. Then $G$ has four distinct $\mathcal{L}$-eigenvalues if and only if $G=C_6$, $G=C_7$, $G=U_2(1)$ or $G=U_4(1,1,1)$.
\end{thm}
\begin{proof}
Let $G$ be a unicyclic graph with four distinct $\mathcal{L}$-eigenvalues. Suppose $\mathrm{Spec}_{\mathcal{L}}(G)=\{[\alpha]^{m_1},[\beta]^{m_2},[\gamma]^{m_3},[0]^1\}$, where $2\geq \alpha>\beta>\gamma>1$. Then the diameter of  $G$ is equal to $2$ or $3$ due to $G$ cannot be a complete graph. Thus $G$ must be one of the graphs (excluding $U_1=K_3$) shown in Fig. \ref{fig-1}. If $G\in\{U_7,U_{10},U_{13},U_{14}\}$, then we have $G=U_{13}=C_6$ or $G=U_{14}=C_7$ by simple computation. To prove the result, it suffices to consider the remain cases.

First suppose that $G=U_2(a)$. If $a\geq 2$, then $u_1$ and $u_2$ (see Fig. \ref{fig-1}) have the same neighborhood, and thus $1$ is an $\mathcal{L}$-eigenvalue of $G$ by Lemma \ref{lem-2-3} (i). Then $\beta=1$ or $\gamma=1$ by Lemma \ref{lem-2-2} (v). Without loss of generality, we assume that $\gamma=1$. Putting $u=u_1$, $u=u_0$ and $u=v_0$ (see Fig. \ref{fig-1}) in (\ref{equ-9}), respectively, we obtain the following three equalities:  
\begin{equation*}
\left\{
\begin{aligned}
&(\alpha+\beta-2)\cdot \frac{1}{a+2}=\frac{\alpha\beta}{2m},\\
&\frac{1}{2\cdot2}\cdot 2+(\alpha+\beta-2)(a+\frac{1}{2}\cdot 2)=\frac{\alpha\beta}{2m}(a+2)^2,\\
&\frac{1}{2\cdot(a+2)}\cdot 2+(\alpha+\beta-2)(\frac{1}{a+2}+\frac{1}{2})=\frac{\alpha\beta}{2m}\cdot2^2.\\
\end{aligned}
\right.
\end{equation*}
By simple computation we obtain  $a=0$, which is contrary to $a\geq 2$. If $a=1$, then $G=U_2(1)$ with $\mathrm{Spec}_{\mathcal{L}}(U_2(1))=\{[1.7287], [1.5000],[0.7713],[0]\}$, as required.

Suppose $G=U_3(a,b)$. If $a\geq 2$ or $b\geq 2$, then $1$ is an $\mathcal{L}$-eigenvalue of $G$. As above, assume that $\gamma=1$. Putting $u=u_1$, $u=v_1$, $u=u_0$ and $u=w_0$ in (\ref{equ-9}) one by one, we obtain the following four equalities:
\begin{equation*}
\left\{
\begin{aligned}
&(\alpha+\beta-2)\cdot \frac{1}{a+2}=\frac{\alpha\beta}{2m},\\
&(\alpha+\beta-2)\cdot \frac{1}{b+2}=\frac{\alpha\beta}{2m},\\
&\frac{1}{2(b+2)}\cdot 2+(\alpha+\beta-2)(a+\frac{1}{b+2}+\frac{1}{2})=\frac{\alpha\beta}{2m}(a+2)^2,\\
&\frac{1}{(a+2)(b+2)}\cdot 2+(\alpha+\beta-2)(\frac{1}{a+2}+\frac{1}{b+2})=\frac{\alpha\beta}{2m}\cdot2^2,\\
\end{aligned}
\right.
\end{equation*}
from which one can easily deduce that $a=b=0$, a contradiction. If $a=b=1$, then $\mathrm{Spec}_{\mathcal{L}}(G)=\mathrm{Spec}_{\mathcal{L}}(U_3(1,1))=\{[1.7676],[1.6667],[1],[0.5657],[0]\}$, contrary to our assumption.

Now suppose that $U=U_4(a,b,c)$. If $\max\{a,b,c\}\geq 2$, as above, putting $u=u_1$, $u=v_1$ and $u=w_1$ in (\ref{equ-9}) one by one, we obtain
\begin{equation*}
\left\{
\begin{aligned}
&(\alpha+\beta-2)\cdot \frac{1}{a+2}=\frac{\alpha\beta}{2m},\\
&(\alpha+\beta-2)\cdot \frac{1}{b+2}=\frac{\alpha\beta}{2m},\\
&(\alpha+\beta-2)\cdot \frac{1}{c+2}=\frac{\alpha\beta}{2m},\\
\end{aligned}
\right.
\end{equation*}
which implies that $a=b=c$, and so $G=U_4(a,a,a)$. Then it suffices to consider the $\mathcal{L}$-spectrum of the graph $U_4(a,a,a)$, where $a\geq 1$. In fact, the $\mathcal{L}$-polynomial of $U_4(a,a,a)$ is equal to
\begin{align*}
&~~~~~P_\mathcal{L}(U_4(a,a,a))\\
&=|\lambda I-\mathcal{L}|\\
&=\left|\begin{matrix}
(x-1)I_a&0&0&\frac{1}{\sqrt{a+2}}\mathbf{j}_a&0&0\\
0&(x-1)I_a&0&0&\frac{1}{\sqrt{a+2}}\mathbf{j}_a&0\\
0&0&(x-1)I_a&0&0&\frac{1}{\sqrt{a+2}}\mathbf{j}_a\\
\frac{1}{\sqrt{a+2}}\mathbf{j}_a^T&0&0&x-1&\frac{1}{a+2}&\frac{1}{a+2}\\
0&\frac{1}{\sqrt{a+2}}\mathbf{j}_a^T&0&\frac{1}{a+2}&x-1&\frac{1}{a+2}\\
0&0&\frac{1}{\sqrt{a+2}}\mathbf{j}_a^T&\frac{1}{a+2}&\frac{1}{a+2}&x-1\\
\end{matrix}
\right|\\
&=\left|\begin{matrix}
(x-1)I_a&0&0&0&0&0\\
0&(x-1)I_a&0&0&0&0\\
0&0&(x-1)I_a&0&0&0\\
\frac{1}{\sqrt{a+2}}\mathbf{j}_a^T&0&0&\frac{(a+2)(x-1)^2-a}{(a+2)(x-1)}&\frac{1}{a+2}&\frac{1}{a+2}\\
0&\frac{1}{\sqrt{a+2}}\mathbf{j}_a^T&0&\frac{1}{a+2}&\frac{(a+2)(x-1)^2-a}{(a+2)(x-1)}&\frac{1}{a+2}\\
0&0&\frac{1}{\sqrt{a+2}}\mathbf{j}_a^T&\frac{1}{a+2}&\frac{1}{a+2}&\frac{(a+2)(x-1)^2-a}{(a+2)(x-1)}\\
\end{matrix}
\right|\\
&=\frac{1}{(a+2)^3}x(x-1)^{3a-3}p_1(x)(p_2(x))^2,
\end{align*}
where $p_1(x)=(a+2)x-2(a+1)$ and $p_1(x)=(a+2)x^2-(2a+5)x+3$. It is easy to verify that $1$ cannot be a root of $p_1(x)$ or $p_2(x)$ due to $a\neq 0$. Also, $p_1(x)$ and $p_2(x)$ cannot share the same root because $a\neq 0$. Furthermore, the roots of $p_2(x)$ must be distinct because the discriminant $(2a+5)^2-4\cdot(a+2)\cdot 3=4a^2+8a+1>0$ due to $a\geq 1$. Therefore, $G=U_4(a,a,a)$ has four distinct $\mathcal{L}$-eigenvalues if and only if $a=1$.

Next suppose that $U=U_5(a)$. If $a\geq 2$, as above, putting $u=u_1$ and $u=u_0$ in (\ref{equ-9}), we have
\begin{equation*}
\left\{
\begin{aligned}
&(\alpha+\beta-2)\cdot \frac{1}{a+1}=\frac{\alpha\beta}{2m},\\
&(\alpha+\beta-2)\cdot (a+\frac{1}{3})=\frac{\alpha\beta}{2m}\cdot (a+1)^2,
\end{aligned}
\right.
\end{equation*}
from which one can deduce that $a+1=a+\frac{1}{3}$, a contradiction. If $a=1$, then $\mathrm{Spec}_{\mathcal{L}}(G)=\mathrm{Spec}_{\mathcal{L}}(U_5(1))=\{[1.8566],[1.5000],[1.2975],[0.3459],[0]\}$,  a contradiction.

Suppose $G=U_6(a,b)$. If $a\geq 2$ or $b\geq 2$, as above, putting $u=u_1$ and $u=u_0$ in (\ref{equ-9}) in turn, we get
\begin{equation*}
\left\{
\begin{aligned}
&(\alpha+\beta-2)\cdot \frac{1}{a+1}=\frac{\alpha\beta}{2m},\\
&(\alpha+\beta-2)\cdot (a+\frac{1}{b+3})=\frac{\alpha\beta}{2m}\cdot (a+1)^2,
\end{aligned}
\right.
\end{equation*}
implying that $b=-2$, contrary to $b\geq 1$. If $a=b=1$, then $\mathrm{Spec}_{\mathcal{L}}(G)=\mathrm{Spec}_{\mathcal{L}}(U_6(1,1))=\{[1.8762],[1.5000]^2,[0.7838],[0.3400],[0]\}$, which is impossible.

Now suppose that $G\in\{U_8(a),U_9(a, b)\}$. Then $G$ is a bipartite graph. We claim that $a=b=1$, since otherwise $G$ cannot have four distinct $\mathcal{L}$-eigenvalues by Corollary \ref{cor-2-1}. Thus $G=U_8(1)$ or $G=U_9(1,1)$, which are also impossible due to $\mathrm{Spec}_{\mathcal{L}}(U_8(1))=\{[2],[1.4082],[1],[0.5918],[0]\}$ and $\mathrm{Spec}_{\mathcal{L}}(U_9(1,1))=\{[2],[1.5000],[1.3333],[0.6667],$ $[0.5000],[0]\}$.

Suppose $G=U_{11}(a)$. If $a\geq 2$, putting the $u=u_1$ and  $u=u_0$ in (\ref{equ-9}), we have
\begin{equation*}
\left\{
\begin{aligned}
&(\alpha+\beta-2)\cdot \frac{1}{a+2}=\frac{\alpha\beta}{2m},\\
&(\alpha+\beta-2)\cdot (a+\frac{1}{2}\cdot 2)=\frac{\alpha\beta}{2m}\cdot (a+2)^2.
\end{aligned}
\right.
\end{equation*}
Thus we deduce that $a+2=a+1$, a contradiction. If $a=1$, then $\mathrm{Spec}_{\mathcal{L}}(G)=\mathrm{Spec}_{\mathcal{L}}(U_{11}(1))=\{[1.8691],[1.8090],[1.1759],[0.6910],[0.4550],[0]\}$, contrary to our assumption.

Suppose $G=U_{12}(a,b)$. If $a\geq 2$ or $b\geq 2$, putting $u=u_1$ and $u=u_0$ in (\ref{equ-9}), we obtain the following three equalities:
\begin{equation*}
\left\{
\begin{aligned}
&(\alpha+\beta-2)\cdot \frac{1}{a+2}=\frac{\alpha\beta}{2m},\\
&(\alpha+\beta-2)\cdot (a+\frac{1}{2}+\frac{1}{b+2})=\frac{\alpha\beta}{2m}\cdot (a+2)^2,
\end{aligned}
\right.
\end{equation*}
which implies that $b=-\frac{4}{3}$, a contradiction. Then $a=b=1$, then $\mathrm{Spec}_{\mathcal{L}}(G)=\mathrm{Spec}_{\mathcal{L}}(U_{12}(1,1))=\{[1.8931],[1.8259],[1.3766],[1],[0.4642],[0.4402],[0]\}$, a contradiction.

We complete the proof.
\end{proof}

\end{document}